\documentclass[11pt,reqno]{amsart}
\usepackage{amsmath,amsthm,amssymb,amsrefs,enumerate,mathtools,verbatim,stmaryrd,xcolor,microtype,graphicx,aliascnt,mathrsfs}
\usepackage[bookmarksdepth=2,linktoc=page,colorlinks,linkcolor={red!70!black},citecolor={red!80!black},urlcolor={blue!80!black},pdfpagemode=UseOutlines,pdfstartview={fitH}]{hyperref}   
\usepackage[T1]{fontenc}
\usepackage[utf8]{inputenc}
\usepackage[english]{babel} 
\usepackage[top=2.5cm,bottom=3.3cm,left=2.8cm,right=2.8cm]{geometry}
\usepackage[mathcal]{euscript}                               
\usepackage{cleveref}
\usepackage{xr}

\usepackage{microtype}

\usepackage{fancyhdr}

\newtheorem{theorem}{Theorem}[section]

\newtheorem{lemma}[theorem]{Lemma}
\newtheorem{proposition}[theorem]{Proposition}

\theoremstyle{definition}

\theoremstyle{definition}

\newtheorem{fact}[theorem]{Fact}

\theoremstyle{definition}
\newtheorem{definition}[theorem]{Definition}

\theoremstyle{definition}

\theoremstyle{definition}

\theoremstyle{definition}

\newcommand\N{{\mathbb N}}

\newcommand\R{{\mathbb R}}

\newcommand\Vol{{\operatorname{Vol}}}
\newcommand{\la}{\langle}
\newcommand{\ra}{\rangle}

\renewcommand\geq{\geqslant}
\renewcommand\leq{\leqslant}
\newcommand{\opn}[1]{\operatorname{#1}}

\title{Monotonicity of the deficit in the local log-Brunn-Minkowski inequality}

\author{Shouda Wang}
\address{\rm Fine Hall, Princeton University}
\email{shoudawang@princeton.edu}

\begin{document}

\begin{abstract}
We establish a monotonicity property of the deficit associated with the local log–Brunn–Minkowski inequality (LLBM) under addition of line segments.
As a corollary, if the LLBM holds for a convex body 
$K$, then it also holds for 
$K
+
Z$ for any zonoid 
$Z$, which in particular yields a new proof of the inequality for zonoids.
Moreover, assuming the LLBM is valid, we prove that equality in the LLBM for smooth convex bodies with 
$C^2$support functions occurs only for homothetic bodies.
\end{abstract}

\maketitle

\section{Introduction}

The classical Brunn--Minkowski inequality is a fundamental tool in convex geometry asserting the log-concavity of the volume functional: for convex bodies
\(K,L\subset\R^{n}\) and \(t\in[0,1]\),
\[
  \Vol \bigl(tK+(1-t)L\bigr)
  \ \ge\
  \Vol(K)^{t}\,\Vol(L)^{1-t}.
\]
It is an important object of study in various  areas of geometry, analysis, and combinatorics, see, for example,  \cite{Gardner}, \cite{Figalli}, and \cite{TaoVu}.
As a generalization of the classical Brunn-Minkowski inequality, Böröczky, Lutwak, Yang, and Zhang \cite{BLYZ} conjectured the \emph{log-Brunn-Minkowski inequality}: for  origin-symmetric convex bodies $K,L$,
\[
  \Vol\!\left(K^{t}L^{1-t}\right)
  \ \ge\
  \Vol(K)^{t}\,\Vol(L)^{1-t},
\]
 where $K^{t}L^{1-t}$ is defined later in Section \ref{sec2}. 
The log-Brunn-Minkowski inequality also has many interesting connections to other problems. Indeed, it is related to the uniqueness of  Minkowski's problem for cone volume measures \cite{BLYZ}, to a spectral gap estimate for the Hilbert operator asociated to symmetric convex bodies \cite{KolesnikovMilman}, to an inequality for the Gaussian density \cite{Saroglou}, just to name a few. We refer to the survey \cite[Sec.3 \& 4]{BoroczkySurvey} for more detailed discussions.

\medskip
\noindent\textbf{Equivalent formulations.}
Since the classical Brunn-Minkowski inequality is a statement about the concavity of the function $t\mapsto \log \Vol (tK+(1-t)L)$, one naturally looks at the first and second order derivatives of this function. 
To this end, it is necessary to compute the derivatives of the volume functional with respect to  convex bodies, which is given by their \emph{mixed volume}: for convex bodies   $C_1,...,C_n\subseteq \R^n$, their mixed volume $V(C_1,...,C_n)$ is defined as  
$$
V(C_1,...,C_n):=\frac{1}{n!}\frac{d}{dt_1}\Bigr|_{t_1=0}\cdots \frac{d}{dt_n}\Bigr|_{t_n=0}
\Vol (t_1C_1+\cdots +t_nC_n) .
$$
See \cite[Chap.5]{Schneider} for a systematic treatment of mixed volumes.

It turns out that by taking the first and second derivatives of $t\mapsto \log \Vol (tK+(1-t)L)$, we obtain  two new inequalities called Minkowski's first and second  inequality respectively.  It naturally follows that these two inequalities and the Brunn-Minkowski inequality are formally equivalent to each other (see \cite[Sec.7.2]{Schneider}  or \cite[Lem.1.1]{LLBMzonoid}).
\begin{align*}
&\cdot\text{Brunn-Minkowski (0th order): } V(tK+(1-t)L)\geq V(K)^{t}V(L)^{1-t},\\
  &\cdot  \text{Minkowski's first (1st order): }V(K,...,K,L)\geq V(K)\left(\frac{V(L)}{V(K)}\right)^{\frac{1}{n}},\\
&\cdot \text{Minkowski's second (2nd order): }V(K,...,K,L)^2\geq V(K)V(K,...,K,L,L).
\end{align*}
Similar to the case of the classical Brunn-Minkowski inequality, one can take the first and second order derivatives of the function $t\mapsto \log V(K^{1-t}L^t)$, thus obtaining the \emph{log-Minkowski inequality} and the \emph{local log-Brunn-Minkowski inequality}.
\begin{align*}
&\cdot\text{log-Brunn-Minkowski (0th order): } V(K^{t}L^{1-t})\geq V(K)^{t}V(L)^{1-t},\\
  &  \cdot\text{log-Minkowski (1st order): }V(K,...,K, h_K \log\tfrac{h_L}{h_K})\geq V(K) \log\left(\frac{V(L)}{V(K)}\right)^{\frac{1}{n}},\\
& \cdot\text{local log-Brunn-Minkowski (2nd order): }\\
&  \qquad  V(K,...,K,L)^2\geq (1-\tfrac{1}{n})V(K) V(K,...,K,L,L)+\tfrac{1}{n} V(K)V(K,...,K,\tfrac{h_L^2}{h_K}).
\end{align*}
Here $V(K,...,K,f)$ is to be understood as  $\tfrac{1}{n}\int_{S^{n-1}}  f dS_K$, where $S_K$ is the area measure of $K$.
The equivalence of these three inequalities is much harder to establish  than  the classical setting due to the non-linearity of the map $t\mapsto K^{1-t}L^t$. Nevertheless, the equivalence of the log-Brunn-Minkowski and the log-Minkowski inequality was shown in \cite{BLYZ} and the equivalence of the local log-Brunn-Minkowski inequality to the the first two inequalities was shown by combining \cite{KolesnikovMilman} with \cite{Putterman}. Therefore, the study of the log-Brunn-Minkowski (log-BM) conjecture is  reduced to the study of the local log-Brunn-Minkowski (LLBM) conjecture, which is the focus of the current note.

\medskip
\noindent\textbf{Prior work.}
The dimension 2 case of the log-BM conjecture was settled in the original paper  \cite{BLYZ}, while the conjecture  in complete generality stays wide open in higher dimensions. 

The log-BM conjecture has also been verified when both $K$ and $L$ satisfy certain symmetry assumptions. Saroglou \cite{Saroglou} proved log-BM when $K$ and $L$ are both unconditional. Böröczky and Kalantzopoulos \cite{BoroczkyK} proved  log-BM for convex bodies invariant under reflections through $n$ independent hyperplanes. Rotem \cite{Rotem} verified the log-BM for complex convex bodies. 

Note that these results put the same restriction on $K$ and $L$, which is natural for  log-BM since log-BM  is symmetric about $K$ and $L$ up to  replacing $t$ by $1-t$. However, for the log-Minkowski and LLBM inequality,
$K$ and $L$ play a very different role. In particular, 
Kolesnikov and Milman \cite[Cor.5.4]{KolesnikovMilman} pointed out that the LLBM for a fixed $K$ and all $L$ is equivalent to a lower bound on the spectral gap of the  Hilbert operator associated to $K$. Therefore it is  natural to think of $L$ merely as a ``test body'' and $K$ as the ``base body''.
And then it is  natural to ask for a  class of base bodies $K$ such that log-Minkowski/LLBM holds for all  test bodies $L$.

In \cite{LLBMzonoid},  this spectral interpretation was used to prove  LLBM for zonoids, from which it follows that the log-Minkowski inequality also holds for zonoids $K$ and all $L$. The equality cases of LLBM for zonoids were also characterized in \cite{LLBMzonoid}. 
Xi \cite{Xi2024} provided another proof of the log-Minkowski inequality for zonoids by proving a reverse log-Minkowski inequality and showing a reverse-to-forward principle. This approach  also allowed him to characterize the equality cases of the log-Minkowski inequality for zonoids.

\medskip
\noindent\textbf{Main ideas of this paper.}
The starting point of this note is to provide a different proof of LLBM for zonoids. Our approach differs from the two previous results in that we work directly with the deficit of LLBM and prove that the deficit is monotone with respect to adding an interval to the base body $K$ under a certain normalization condition on $L$.  

Let $K$ be a convex body containing the origin in its interior and $f$ be a difference of two support functions. We define the \emph{deficit} of LLBM as
\[
  \Delta(K,f)
  \ :=\
  \frac{V\!\big(K[n-1],f\big)^{2}}{\Vol(K)}
  \;-\;\frac{n-1}{n}\,V\!\big(K[n-2],f,f\big)
  \;-\;\frac{1}{n}\,V\!\Big(K[n-1],\,\frac{f^{2}}{h_{K}}\Big).
\] 
Here mixed volume of functions is defined by multilinearity (see Section \ref{sec2}).
Then the LLBM inequality asserts   \(\Delta(K,h_L)\ge 0\) holds for any origin-symmetric  body $L$.
By a standard approximation argument \cite[Appendix A]{Putterman},  it can be shown that \(\Delta(K,h_L)\ge 0\) for all origin-symmetric bodies $L$ is equivalent to  \(\Delta(K,f)\ge 0\) for all differences of support functions of origin-symmetric  bodies $f$.
In the sequel, we will work with these two formulations of LLBM interchangeably.

Let us first record an elementary  observation of the deficit of LLBM.
\begin{lemma}\label{Lem1}
Let $K$ be a full dimensional convex body  and $f$ a difference of support functions, then
    $\Delta(K,f)=\Delta(K,f+ch_K)$ holds for all $c\in\R$.
\end{lemma}
\begin{proof}
By a straightforward computation,
    \begin{align*}
    \Delta(K,f+ch_K) =&\frac{(V(K[n-1],f)+cV(K))^2}{V(K)}-\frac{n-1}{n} V(K[n-2],f,f)\\
    &-2c\frac{n-1}{n} V(K[n-1],f)-c^2\frac{n-1}{n} V(K)\\
    &-\frac{1}{n} V(K[n-1], \tfrac{f^2}{h_K}+2cf+c^2h_K) =\Delta(K,f).
    \end{align*}
\end{proof}
Our key tool is an explicit formula for the derivative of $\Delta(K_t,f)$, where $K_t=K+tI$, $t\geq 0$. We assume $K$ contains the origin in its interior,  $I$ is a segment  containing  the origin, and $f$ is a difference of support functions.
Then by a straightforward computation,
    \begin{align}\nonumber
        \frac{d}{dt}&\Delta(K_t,f)\\
        \nonumber
        =&
        2(n-1)\frac{V(K_t[n-2],I,f)V(K_t[n-1],f)}{V(K_t)}-n\left(\frac{V(K_t[n-1],f)}{V(K_t)}\right)^2 V(K_t[n-1],I)\\
        \nonumber
        & -\frac{(n-1)(n-2)}{n} V(K_t[n-3],I,f,f)-\frac{n-1}{n} V(K_t[n-2],I,\tfrac{f^2}{h_{K_t}})+\frac{1}{n} V(K_t[n-1],\tfrac{f^2}{h_{K_t}^2}h_I)\\
        \nonumber
        =&  \frac{(n-1)^2}{n} \left(  \frac{V(K_t[n-2],I,f)^2}{V(K_t[n-1],I)} -\frac{n-2}{n-1} V(K_t[n-3],I,f,f)-\frac{1}{n-1} V(K_t[n-2],I,\tfrac{f^2}{h_{K_t}}) \right)\\
        \nonumber
        &- n V(K_t[n-1],I)\left( \frac{V(K_t[n-1],f)}{V(K_t)} - \frac{n-1}{n} \frac{V(K_t[n-2],I,f)}{V(K_t[n-1],I)}\right)^2 \\ 
        \nonumber
        &+\frac{1}{n} V(K_t[n-1],\tfrac{f^2}{h_{K_t}^2}h_I)\\
        \nonumber
        =&  \frac{(n-1)^2}{n^2}|I| \Delta(P_{I^\perp} K_t,f|_{I^\perp}) 
        +\frac{1}{n} V(K_t[n-1],\tfrac{f^2}{h_{K_t}^2}h_I)
        \\
        &- n V(K_t[n-1],I)\left( \frac{V(K_t[n-1],f)}{V(K_t)} - \frac{n-1}{n} \frac{V(K_t[n-2],I,f)}{V(K_t[n-1],I)}\right)^2 , 
        \label{Eqn2bis}
    \end{align}
    where we used the projection formula \eqref{Eqn20} for the last equality.
 The three terms on the RHS of this identity are respectively:
\begin{enumerate}
  \item $\frac{(n-1)^2}{n^2}|I| \Delta(P_{I^\perp} K_t,f|_{I^\perp})$, which is nonnegative by induction hypothesis since  $\Delta(P_{I^\perp} K_t,f|_{I^\perp})$ is  an  LLBM term in one lower dimension;
  \item $\frac{1}{n} V(K_t[n-1],\tfrac{f^2}{h_{K_t}^2}h_I)$,  which is  nonnegative since $\tfrac{f^2}{h_{K_t}^2}h_I\geq 0$;
  \item $- n V(K_t[n-1],I)\left( \frac{V(K_t[n-1],f)}{V(K_t)} - \frac{n-1}{n} \frac{V(K_t[n-2],I,f)}{V(K_t[n-1],I)}\right)^2$, which can be set to zero by replacing $f$ with \( f+ch_{K_{t}}\) for some suitable $c$ because of Lemma \ref{Lem1}.
\end{enumerate}
This yields a monotonicity principle for the deficit of LLBM under segment addition.
On the other hand, by a simple application of the Alexandrov-Fenchel inequality, one can verify LLBM for the cube (and hence all parallelotopes since LLBM is covariant under linear transformation). Now, start with a parallelotope and apply the monotonicity principle  iteratively  to different intervals $I$, we conclude that LLBM holds for all zonotopes. Since zonoids are limits of zonotopes and mixed volume is a continuous functional on the space of convex bodies \cite[Thm.5.1.7]{Schneider}, LLBM holds for all zonoids.  We carry out this   approach  in Section \ref{sec3} to exhibit a self-contained new proof of LLBM for   zonoids.

Beyond reproving LLBM for zonoids, the segment-addition monotonicity also immediately implies that the Minkowski sum of a convex body satisfying LLBM with any zonoid must also satisfy LLBM.
\begin{theorem}\label{Thm3}
    Let  $K\in \mathcal{K}^n_s$.
    Assume   LLBM holds for any  projection $P_E K$ of $K$:  $\Delta(P_E K,h_C)\geq 0$ holds for any  linear subspace $E\subseteq \R^n$  and any  $C\in \mathcal{K}^{\dim E}_s$.     
    Then for any origin-symmetric zonoid $Z\subseteq \R^n$ and any $L\in \mathcal{K}^n_s$, $\Delta(K+Z,h_L)\geq 0$.
\end{theorem}
Note that $E=\R^n$ is allowed in the statement. That is, as part of the hypothesis we assume $K$ satisfies LLBM.

In a different direction, equation \eqref{Eqn2bis} can  be adapted to study the equality cases of LLBM assuming LLBM is true. We show that for smooth bodies with $C^2$ support functions, only trivial equality cases can occur. In a paper \cite{WeiyongJunbang} which appeared at around the same time as the current one, a similar result was obtained. While both their and our results  are based on calculus of variations of  the deficit of LLBM, their proof proceeds with a PDE method while our proof is  geometric. Our approach also  only requires a weaker assumption on the regularity of the body $K$.
\begin{theorem}
\label{Thm2}
If \(K \in \mathcal{K}^n_s\)  is smooth with $C^2$ support function and \(L\) is  origin-symmetric   with $C^2$ support function satisfying \(\Delta(K, h_L) = 0\), 
then \(L\) is a dilate of \(K\); that is, \(L = cK\) for some \(c \ge 0\).
\end{theorem}

\medskip
\noindent\textbf{Organization.}
We recall the  preliminaries in Section \ref{sec2},  prove the LLBM for zonoids as well as Theorem \ref{Thm3} in Section \ref{sec3}, and prove Theorem \ref{Thm2} in Section \ref{sec4}.

\section*{Acknowledgement}
The author is grateful to his advisor Ramon van Handel for helpful discussions. He also thanks Leo Brauner for a suggestion on Theorem \ref{Thm3}, and thanks Weiyong He and Junbang Liu for communicating their results and interesting discussions.

\section{Preliminaries}
\label{sec2}
\subsection{}
A convex body in $\R^n$ is a convex compact subset of $\R^n$.
The \emph{support function} of a convex body $K$ is defined by  $h_K(u):=\max_{x\in K} \la x,u\ra$. 
The Minkowski sum of $K$ with $L$ is defined as $K+L:=\{a+b:a\in K, b\in L\}$.
The Minkowski difference of $K$ and $L$ is defined as $K\div L:=\{x\in\R^n: x+L\subseteq K\}$.
The Minkowski sum can be expressed via  support functions as
$$
K+L=\{x\in\R^n: \forall u\in S^{n-1}, \,\la x,u\ra \leq h_K(u)+h_L(u) \}.
$$
This motivates the definition of $K^tL^{1-t}$. Indeed, as above we can write the weighted arithmetic mean $tK+(1-t)L$ as 
$$
tK+(1-t)L=\{x\in\R^n: \forall u\in S^{n-1}, \,\la x,u\ra \leq th_K(u)+(1-t)h_L(u) \}.
$$
Hence similarly, for  $K$ and $L$ containing the origin in their interior (i.e. $h_K(u),h_L(u)>0$ for all $u\neq 0$), it is  natural to define the weighted geometric mean of $K$ and $L$ as 
$$
K^{t}L^{1-t}:=\{x\in\R^n: \forall u\in S^{n-1}, \,\la x,u\ra \leq h_K(u)^th_L(u)^{1-t}\}.
$$
Denote by $\mathcal{K}^n_s$ the set of all origin symmetric convex bodies with nonempty interior.
We will only consider the log-BM, log-Minkowski, LLBM for convex bodies in the class $\mathcal{K}^n_s$.
\subsection{}
Let us collect some facts about mixed volume that will be used in this note, all of which can be found in \cite[Chap.5]{Schneider}. 
First, the mixed volume can be computed via integrating the support function against the mixed area measure 
\begin{equation}\label{Eqn13}
    V(C_1,...,C_n)=\frac{1}{n}\int_{S^{n-1}}h_{C_n}dS_{C_1,...,C_{n-1}}.
\end{equation}
In the case where $C_1=\cdots =C_{n-1}=C$, $S_{C,...,C}$ is equal to $S_C$, the area measure of $C$. In view of \eqref{Eqn13}, we naturally define the mixed volume of  $f\in C(S^{n-1};\R)$ and convex bodies $C_1,...,C_{n-1}$, 
\begin{equation}\label{Eqn19}
V(f,C_1,...,C_{n-1}):=\frac{1}{n}\int_{S^{n-1}} fdS_{C_1,...,C_{n-1}}.
\end{equation}
Under a linear transformation, the mixed volume  behaves in the same way as the usual volume: for $A\in \operatorname{GL}(n,\R)$ and convex bodies $C_1,...,C_n$,
\begin{equation}\label{Eqn28}
    V(A(C_1),...,A(C_n))=|\det(A)| V(C_1,...,C_n).
\end{equation}
Let $I$ be an interval segment.
Denote by $I^\perp$ the orthogonal complement of the direction of $I$ and let $P_{I^\perp}$ be the orthogonal projection onto $I^\perp$.
Then $h_{P_{I^\perp}C}(u) = h_{C}(P_{I^\perp} u)$ for all $u\in \R^n$. 
In particular, $h_{P_{I^\perp}C} = h_{C}|_{I^\perp}$.
The projection formula \cite[Theorem 5.3.1]{Schneider} asserts that 
\begin{equation}
    \label{Eqn20}
V(I, C_1,...,C_{n-1})=\frac{|I|}{n}V(P_{I^\perp}C_1,...,P_{I^\perp}C_{n-1})
\end{equation} 
where $|I|$ is the length of $I$.
For $f\in C(S^{n-1};\R)$, the  projection formula  extends to 
$$
V(I,f, C_1,...,C_{n-2})=\frac{|I|}{n}V(f|_{I^\perp\cap S^{n-1}},P_{I^\perp}C_1,...,P_{I^\perp}C_{n-2}),
$$

\subsection{}

Mixed volumes and mixed area measures  can be extended from convex bodies to a certain class of functions by multilinearity \cite[Sec. 5.2]{Schneider}. We denote by $D(S^{n-1})$ the linear space spanned by  restrictions of support functions to $S^{n-1}$. And denote by $D_s(S^{n-1})=\{f\in D(S^{n-1})\mid \forall x\in S^{n-1},\, f(x)=f(-x)\}$ the subspace in $D(S^{n-1})$ of origin-symmetric functions. Any $f\in D_s(S^{n-1})$ can be written as $f=(h_K-h_L)|_{S^{n-1}}$ with $K,L\in \mathcal{K}^n_s$. For such $f$, 
\begin{align*}
    V(f,C_1,...,C_{n-1})=&V(K,C_1,...,C_{n-1})-V(L,C_1,...,C_{n-1}),\\
    S_{f,C_1,...,C_{n-2}}=&S_{K,C_1,...,C_{n-1}} -S_{L,C_1,...,C_{n-1}}.
\end{align*}
Note that the first formula is compatible with \eqref{Eqn19}. 
We also have
$$
V(f,f,C_1,...,C_{n-2})=V(K,f,C_1,...,C_{n-2})-V(L,f,C_1,...,C_{n-2}).
$$

\section{LLBM for Zonoids}
\label{sec3}

In this section we use \eqref{Eqn2bis} to give a completely self-contained proof of LLBM for zonoids.
The proof will proceed by a double induction. We induct on the dimension, and then for a fixed dimension we induct on the number of interval summands. The latter  induction step is done in Lemma \ref{Lem2} using the formula \eqref{Eqn2bis}. Lemma \ref{Lem2} also immediately implies Theorem \ref{Thm3}. The former induction step is done in Proposition \ref{Prop1}.

\subsection{Induction on number of interval summands}

\begin{lemma}\label{Lem2}
Let $K\in \mathcal{K}^n_s$ ($n\geq 2$) and let $I$ be a segment centered at the origin such that $K=(K\div I)+I$ and $K\div I\in \mathcal{K}^n_s$. 
    Assume $\Delta(P_{I^\perp}K,g)\geq 0$ for all $g\in D_s(S^{n-2})$. Then for all  $f\in D_s(S^{n-1})$,  we have
    $$
    \Delta(K,f)=\Delta(K,\bar f)\geq \Delta(K\div I,\bar f).
    $$
    where  $\bar f=f-ch_{K}$ and  $ c =n \frac{V(K[n-1],f)}{V(K)} - (n-1)\frac{V(K[n-2],I,f)}{V(K[n-1],I)}$. 
\end{lemma} 
\begin{proof}

By Lemma \ref{Lem1},  $\Delta(K,\bar f)=\Delta(K,f)$ holds.
Denote $K_t=(K\div I)+tI$ for $t\in [0,1]$. 
We have $K=K_t+(1-t)I$.
By the choice of $c$,
\begin{equation}
\frac{V(K[n-1],\bar f)}{V(K)}=\frac{n-1}{n}\frac{V(K[n-2],I,\bar f)}{V(K[n-1],I)}.
    \label{Eqn15}
\end{equation}
Hence, for any $t\in[0,1]$, 
\begin{align}
\nonumber
\frac{V(K_t[n-1],\bar f)}{V(K_t)}
&=\frac{V(K[n-1],\bar f)-(n-1)(1-t)V(K[n-2],I,\bar f)}{V(K)-n(1-t)V(K[n-1],I)}\\
=&\frac{n-1}{n}\frac{V(K[n-2],I,\bar f)}{V(K[n-1],I)}=\frac{n-1}{n}\frac{V(K_t[n-2],I,\bar f)}{V(K_t[n-1],I)},
\label{Eqn16}
\end{align}
where for the first equality we used $K=K_t+(1-t)I$ and $V(I,I,\cdot,...,\cdot)=0$, for the second equality we used \eqref{Eqn15}, for the third equality we used $V(I,I,\cdot,...,\cdot)=0$ again.
Note that $\frac{d}{dt}h_{K_t} = h_I$, \eqref{Eqn2bis} yields
 \begin{align}
        \frac{d}{dt}&\Delta(K_t,\bar f)=    \frac{(n-1)^2}{n^2} |I|D(P_{I^\perp} K_t, f|_{I^\perp})
        \label{Eqn3}
        +\frac{1}{n} V(K_t[n-1],\tfrac{f^2}{h_{K_t}^2}h_I),
    \end{align}
    where we  used \eqref{Eqn16}.
    Since $\tfrac{f^2}{h_{K_t}^2}h_I\geq 0$, we have  $V(K_t[n-1],\tfrac{f^2}{h_{K_t}^2}h_I)\geq 0$.
    By  assumption we have  $D(P_{I^\perp} K_t, f|_{I^\perp}) =D(P_{I^\perp} K, f|_{I^\perp}) \geq 0$.
    Hence $\frac{d}{dt}\Delta(K_t,\bar f)\geq 0$.
    Therefore $\Delta(K,f) = \Delta(K,\bar f)\geq \Delta(K\div I,\bar f)$.

\end{proof}
\begin{proof}[Proof of Theorem \ref{Thm3}]
Let us prove the claim by induction on the dimension $n$. Since LLBM is true in dimension 1, our claim is \textit{a fortiori} true in dimension 1. Now assume the claim holds in dimension $n-1\geq 1$ and let us prove the claim in dimension $n$.

Consider  $\Delta (K+I_1+\cdots+I_\ell,f)$, where $I_1,...,I_\ell$ are origin-symmetric segments  and  $f\in D_s(S^{n-1})$. 
       For any subspace $E$ of $I_\ell^\perp$, $\Delta(P_{E}(P_{I_\ell^\perp} K), h_C)=\Delta(P_{E} K, h_C)\geq 0$ for any $C\in \mathcal{K}^{\dim E}_s$.
    Therefore, by induction hypothesis, $\Delta(P_{I_\ell^\perp} (K+I_1+\cdots +I_\ell), h_L)\geq 0$ holds for all $L\in \mathcal{K}^{n-1}_s$.
    This allows us to apply Lemma \ref{Lem2} with $K\leftarrow K+I_1+\cdots+I_\ell$, $I\leftarrow I_\ell$, and $f\leftarrow f$, which yields
    $$
   \Delta(K+I_1+\cdots+I_{\ell}, f)
   \geq  \Delta(K+I_1+\cdots+I_{\ell-1},\bar f)
   $$
   for some $\bar f=f-c_\ell h_{K+I_1+\cdots+I_{\ell}} $ with $c_\ell\in \R$.
   We may repeat this process $\ell$ times to obtain 
    $$
   \Delta(K+I_1+\cdots+I_{\ell}, f)
   \geq \Delta(K+I_1+\cdots+I_{\ell-1},\bar f)\geq \cdots \geq  \Delta(K,\tilde f),
   $$
   where $\tilde f=f-c_\ell h_{K+I_1+\cdots+I_{\ell}}- \cdots -c_1 h_{K+I_1}$. Since $\Delta(K,\tilde f)\geq 0$ by assumption, we have 
   $$
   \Delta(K+I_1+\cdots+I_{\ell}, f)
   \geq 0.
   $$
   We have therefore shown that the sum of $K$ with any zonotope satisfies LLBM.  
    Since zonoids are limits of zonotopes and mixed volumes are continuous with respect to Hausdorff convergence \cite[Thm.5.1.7]{Schneider},  $D(K+Z,f)\geq 0$ follows for any origin symmetric zonoid $Z$ and any $f\in D_s(S^{n-1})$, which completes the proof.
\end{proof}

\subsection{New proof for zonoids}
Since we prove by induction on the dimension, it is necessary to check the base case.
\begin{lemma}\label{Lem4}
    Let $K\subseteq \R$ be a 1-dimensional origin symmetric zonotope, and let $f\in D_s(S^0)$. Then $\Delta(K,f)= 0$.
\end{lemma}
\begin{proof}
    Indeed, $K$ is of the form $K=[-a,a]$ for some $a>0$. Hence $\Delta(K,f) = \frac{(f(1)+f(-1))^2}{2a}-\left(\frac{f(1)^2}{a}+\frac{f(-1)^2}{a}\right)
    =0$ because $f(1)=f(-1)$. 
\end{proof}
For a fixed dimension, in order to apply the induction on number of interval summands (Lemma \ref{Lem2}),  we just need to check that LLBM holds for the cube, which  is a well-known fact. Nevertheless, we provide a proof which is based on  formula \eqref{Eqn2bis} for sake of self-completeness.

\begin{lemma}\label{Lem3}
    Let $C$ be the  cube $[-1,1]^n$ and $f\in D_s(S^{n-1})$ ($n\geq 2$). Then $\Delta(C,f)\geq 0$ holds.
\end{lemma}
\begin{proof}
Let $J=\{0\}^{n-1}\times [-1,1]$. 
Fix $\varepsilon\in(0,1)$.
Now apply Lemma \ref{Lem2} to $K=C$   and $I=(1-\varepsilon)J$.  We have 
$\Delta(C, f)\geq \Delta(C\div (1-\varepsilon)J,\bar f)$  for $\bar f  = f-ch_C$ where 
$$
c = n\frac{V(C[n-1],f)}{V(C)}-(n-1)\frac{V(C[n-2],J,f)}{V(C[n-1],J)}.
$$
Note that $\bar f$ is independent of $\varepsilon$.
Let us give the explicit value of $c$.
Let $C_t=C\div J+tJ$ for $t\in [0,1]$. Note that $C_0=[-1,1]^{n-1}\times \{0\}=P_{J^\perp}C$ has zero volume. We have 
\begin{align*}
c =&    n\frac{V(C_0[n-1],f) -(n-1)V(C_0 [n-2],J,f)}{nV(C_0[n-1],J)}+(n-1)\frac{V(C_0[n-2],J,f)}{V(C_0 [n-1],J)} \\
& \quad = \frac{V(C_0[n-1],f) }{V(C_0[n-1],J)},
\end{align*}
where we used $V(J,J,\cdot,\dots,\cdot)=0$.
Since $C_0$ is a $n-1$ dimensional cube,  we have $V(C_0[n-1],f)=\frac{2^{n-1}}{n}(f(e_n)+f(-e_n))=\frac{2^{n}}{n}f(e_n)$ and  $V(C_0[n-1],J) = \frac{2^{n}}{n}$. Hence $c=f(e_n)$. Therefore, $\bar f(-e_n)=\bar f( e_n)=f( e_n)-c h_{C}( e_n)=0$.
It follows that $V(C_0[n-1],\bar f)= 0$.
Therefore, 
\begin{align*}
    \Delta(C_\varepsilon,\bar f )=&\frac{V(C_\varepsilon[n-1],\bar f)^2}{V(C_\varepsilon)}-\frac{n-1}{n } V(C_\varepsilon[n-2],\bar f,\bar f )-\frac{1}{n} V(C_\varepsilon[n-1],\tfrac{\bar f^2}{h_{C_\varepsilon}})\\
    \geq & \frac{1}{n}\frac{V(C_\varepsilon[n-1],\bar f)^2}{V(C_\varepsilon)}-\frac{1}{n} V(C_\varepsilon[n-1],\tfrac{\bar f^2}{h_{C_\varepsilon}})\\
    =&\frac{1}{n}\frac{(V(C_0[n-1],\bar f)+(n-1)\varepsilon V(C_0[n-2],J,\bar f))^2}{V(C_0)+n\varepsilon V(C_0[n-1],J)}-\frac{1}{n} V(C_\varepsilon[n-1],\tfrac{\bar f^2}{h_{C_\varepsilon}})\\
    =& \varepsilon\frac{ (n-1)^2V(C_0[n-2],J,\bar f)^2}{n^2 V(C_0[n-1],J)} -\frac{1}{n} V(C_\varepsilon[n-1],\tfrac{\bar f^2}{h_{C_\varepsilon}}),
\end{align*}
where we used the Alexandrov-Fenchel inequality and expanded the mixed volumes involving $C_\varepsilon=C_0+\varepsilon J$.
Therefore, 
$$
\liminf_{\varepsilon\to 0} \Delta(C_\varepsilon,\bar f)\geq -  \limsup_{\varepsilon\to 0} \frac{1}{n} V(C_\varepsilon[n-1],\tfrac{\bar f^2}{h_{C_\varepsilon}})=0
$$
because $\bar f(\pm e_n)=0$.
By Lemma \ref{Lem2}, $\Delta(C,f)\geq \liminf_{\varepsilon\to 0}\Delta(C_\varepsilon,\bar f)\geq 0$, which completes the proof.

\end{proof}

\begin{proposition}\label{Prop1}
Assume  $D(K,f)\geq 0$ holds for all zonotopes $K\in\mathcal{K}^{n-1}_s$ and $f\in D_s(S^{n-2})$ ($n\geq 2$). Then $D(K,f)\geq 0$ holds for all zonotopes $K\in\mathcal{K}^{n}_s$ and $f\in D_s(S^{n-1})$.
\end{proposition}
\begin{proof}
    Let $K=I_1+\cdots +I_\ell$ with $I_j$ being origin symmetric segments with distinct directions. 
    We induct on the value of $\ell$.
    Since $K$ is $n$-dimensional, $\ell\geq n$.
     
     For the base case $\ell=n$,  $K=A(C)$ for some non-degenerate linear transform $A$, where $C=[-1,1]^n$. We have $\Delta(K,f)=\Delta(A(C),f) = |\det(A)| \Delta(C, f\circ A^{-t})\geq 0$ by \eqref{Eqn28} and  Lemma \ref{Lem3}.

     Assume the conclusion is proved for $\ell-1$, let's prove it for $\ell$.
     Since $P_{I_\ell^\perp}K\in \mathcal{K}^{n-1}_s $ is a  zonotope, by hypothesis we have $\Delta(P_{I_\ell^\perp}K,g)\geq 0$ for all $g\in D_s(S^{n-2})$.  This allows us to apply Lemma \ref{Lem2} to $K$ and $I_\ell$, obtaining $\Delta(K,f)\geq \Delta(K\div I_\ell, \bar f)$ for some $\bar f\in D_s(S^{n-1})$. By the induction hypothesis for $\ell-1$, $\Delta(K\div I_\ell, \bar f)\geq 0$ since $\bar f\in D_s(S^{n-1})$. Therefore $\Delta(K,f)\geq 0$ holds as desired.
\end{proof}

\begin{theorem}
     $\Delta(K,f)\geq 0$
    holds for all zonoids $K\in\mathcal{K}^n_s$ and all $f\in D_s(S^{n-1})$.
\end{theorem}
\begin{proof}
    By a standard approximation argument, it suffices to prove the claim for zonotopes $K\in\mathcal{K}^n_s$. 
    We induct on the dimension $n$. The case $n=1$ is dealt with in Lemma \ref{Lem4} and the induction step is given by Proposition \ref{Prop1}, which completes the proof.
\end{proof}
 
\section{Equality cases for smooth bodies with $C^2$ support functions}
\label{sec4}
In this section, we will assume the LLBM inequality is true and study the  equality cases of LLBM for smooth bodies with $C^2$ support functions.
It turns out that the monotonicity principle proved in Theorem \ref{Thm3} can be adapted to prove that only trivial equality cases appear in this case.

\subsection{}

Recall that a convex body is called smooth if all of its boundary points have a unique unit normal vector.
For two convex bodies $K$, $L$, 
 denote the Minkowski subtraction  by $K\div L:=\{x\in\R^n\mid L+x\subseteq K\}$.
 The following fact is a strengthening of    \cite[Lemma 7.5.4]{Schneider}. Schneider's original statement didn't involve the  uniform convergence part, which  will be necessary in the proof of  Proposition \ref{Prop2}.  Note also that  Fact \ref{Fact1} only assumes $K$ is full dimensional but not $L$, which is  the case  we need.   The proof of Fact \ref{Fact1} follows principally the same lines as Schneider’s proof. 

 \begin{definition}
     For two convex bodies $K,L$, we say that $L$ is adapted to $K$ if for every $x\in \partial K$ there exists $y\in \partial L$ such that $N(K,x)\subseteq N(L,y)$.
 \end{definition}
 Note that if $K$ is smooth, then any body $L$ is adapted to $K$.
\begin{fact} 
\label{Fact1}
    Let $K,L\subseteq \R^n$ be convex bodies. Assume $K$ is full dimensional and $L$ is adapted to $K$.
    Let 
    $$
    K_t:=\left\{
    \begin{array}{cc}
       K+tL    ,  & t\geq 0, \\
       K\div (-t)L,  & t<0.
    \end{array}
    \right.
    $$
    Then
    $$
    \lim_{t\to 0} \left\|\frac{h_{K_{t}}-h_{K}}{t}-h_L\right\|_{L^\infty(S^{n-1})}=0.
    $$
    In particular, 
     $$
    \frac{d}{dt} \Bigr|_{t=0}h_{K_t}  =  h_L.
    $$
\end{fact}
\begin{proof}
    Since $h_{K_t}=h_K+th_L$ for $t\geq 0$, the desired conclusion is obvious for $t\rightarrow 0+$. Therefore, in the following we only consider the limit $t\to 0-$.
    
    \textbf{Step 1}:
    Let $u\in S^{n-1}$. For all  small $t>0$  such that $K\div tL\neq \emptyset$,  take $z\in F(K\div tL, u)$.
    Let 
    $$
    U_t(u):=\{v\in S^{n-1}\mid h_K(v)=h_{tL+z}(v)\}.
    $$
    We claim that 
    \begin{equation}
    \label{Eqn32}
        u\in \operatorname{pos} U_t(u)
        := \Bigr\{\sum_{i=1}^\ell \lambda_i v_i : \lambda_i\geq 0,\, v_i\in U_t (u), \ell\in \N\Bigr\}.
    \end{equation}
    Indeed, if this were false, then $\operatorname{pos}U_t(u)\setminus \{0\}$ and $u$ are  separated by a hyperplane. That is, there exists $w\in \R^n$ such that $\la w,u\ra>0$ and $\la w, v\ra<0$ for all $v\in U_t(u)\setminus \{0\}$.
    Therefore, $h_K-h_{tL+z}$ is positive on $S^{n-1}\cap\{x:\la x,w\ra\geq 0\}$. In particular, there exists $\varepsilon>0$ such that 
    $$
    h_K(v)-h_{tL+z}(v)\geq \varepsilon \la w,v\ra
    $$
    holds for all $v\in S^{n-1}$. This implies $tL+z+\varepsilon w\subseteq K$. Hence $z+\varepsilon w\in K\div tL $. However, $\la z+\varepsilon w, u\ra >\la z,u\ra$. A contradiction to the choice of $z$.

    \textbf{Step 2}:
    By \eqref{Eqn32} and Carathéodory's theorem, we may choose $v_1,...,v_n\in U_t(u)$ and $\lambda_1,...,\lambda_n\geq 0$ such that $u=\sum_{i=1}^n \lambda_i v_i$.
    We have 
    \begin{align}
    \nonumber
        &\frac{h_K(u)-h_{K\div tL}(u)}{t} -h_L(u)=\frac{h_K(u)-h_{K\div tL+tL}(u)}{t} \\
        \nonumber
        & \qquad\leq \frac{h_K(u)-h_{z+tL}(u)}{t}
        \leq \frac{\sum_{i=1}^n \lambda_i h_K(v_i)-h_{z+tL}(u)}{t}\\
        &\qquad =\frac{\sum_{i=1}^n \lambda_i h_{z+tL}(v_i)-h_{z+tL}(u)}{t}=
        \sum_{i=1}^n \lambda_i h_{L}(v_i)-h_{L}(u),
    \end{align}
    where for the first inequality we used $z\in K\div tL$ and for the second inequality the convexity of support functions.
    Combine this with the fact that $K\supseteq K\div tL+tL$, we obtain 
    $$
    0\leq \frac{h_K(u)-h_{K\div tL}(u)}{t} -h_L(u) \leq   \sum_{i=1}^n \lambda_i h_{L}(v_i)-h_{L}(u).
    $$
    
    \textbf{Step 3}:
    We claim that (note that both $\lambda_i$ and $v_i$ depend on $u$)
    $$
    \lim_{t\to 0} \sup_{u\in S^{n-1}} \Bigr|\sum_{i=1}^n \lambda_i h_{L}(v_i)-h_{L}(u)\Bigr|=0.
    $$
    Assume the converse. Then there exists a number $\alpha>0$, a sequence $(t_j)_{i=1}^\infty$ converging to $0$, a sequence $(u_j)_{i=1}^\infty$ in $S^{n-1}$,
    a sequence $(z_j)_{i=1}^\infty$,  sequences  $(v_{i,j})_{i=1}^\infty$, $i=1,...,n$,  and sequences of positive numbers $(\lambda_{i,j})_{1\leq i\leq n, 1\leq j< \infty}$, such that 
    \begin{align}
    \nonumber
    &z_j\in F(K\div t_jL, u_j), \quad 
        v_{i,j}\in U_{t_j}(u_j)=\{v\in S^{n-1}: h_K(v)=h_{t_jL+z_j}(v)\}\\
       & u_j =   \sum_{i=1}^n \lambda_{i,j} v_{i,j}, \quad  \sum_{i=1}^n \lambda_{i,j} h_{L}(v_{i,j})-h_{L}(u_j)>\alpha.
       \label{Eqn34}
    \end{align}
    Up to passing to a subsequence, we may assume $u_j$ converges to $u\in S^{n-1}$, $z_j$ converges to $z\in F(K,u)$, and $v_{i,j}$ converges to $v_i$ for every $i$.  From $v_{i,j}\in U_{t_j}(u_j)$ we have $h_K(v_i)=h_{z}(v_i)$. Hence $v_i\in N(K,z)$.

    By applying Blaschke selection theorem  to $\operatorname{pos}(U_{t_j}(u_j))\cap S^{n-1}$ we may assume that after passing to a subsequence, 
    $\operatorname{pos}(U_{t_j}(u_j))\cap S^{n-1}$ converges to $P\cap S^{n-1}$, where $P$ is a closed convex cone.

    We claim $P$ does not contain any lineality space. Indeed, assume this is not true, then there exist $\ell\in [n]$ and sequences  $(w_{j}^{k})_{1\leq j< \infty}$  for $1\leq k\leq \ell$, such that $w_{j}^k\in U_{t_j}(u_j)$, $\lim_{j\to \infty} w_j^k=w^k$, and $\operatorname{pos}(w^k,k\in [\ell])$ is equal to the lineality space of $P$. However, from $w_j^k\in U_{t_j}(u_j)$ we have $h_K(w_j^k) = h_{t_j L+z_j } (w_j^k)$. Letting $j$ tend to infinity we obtain $h_K(w^k) = \la z, w^k\ra$, i.e., $w^k \in N(K,z)$. This is a contradiction since $N(K,u)$ does not contain any lineality space due to the full dimensionality assumption on $K$.

    By our previous claim, there exists $b\in S^{n-1}$ and $\delta>0$ such that  $P\subseteq \{v: \la b,v\ra \geq \delta \|v\|\}$. Using this and $u_j=\sum_{i=1}^n \lambda_{i,j} v_{i,j}$, we obtain that for sufficiently large $j$, 
    $$
    1\geq \la b, u_j\ra =\sum_{i=1}^n \lambda_{i,j} \la b, v_{i,j} \ra \geq \delta/2 \sum_{i=1}^n \lambda_{i,j} .
    $$
    Therefore, $\lambda_{i,j} $ are uniformly bounded and we may assume $\lim_{j\to \infty } \lambda_{i,j}=\lambda_i$ by passing to a subsequence.
    By \eqref{Eqn34} we have
    $$
    \sum_{i=1}^n \lambda_i h_L(v_i)-h_L(u)>\alpha.
    $$
    On the other hand,
    by adaptedness we obtain that there exists $y\in \partial L$ such that $N(K,z)\subseteq N(L,y)$.  Since $z\in F(K,u)$ and  $v_i\in N(K,z)$ holds for all $i$, we have $u, v_1,...,v_n\in N(L,y)$.  Hence $\sum_{i=1}^n \lambda_i h_L(v_i)-h_L(u)=0$, a contradiction.

\end{proof}
Another fact we will use follows from the case $m=n-1$ of  \cite[Theorem 4.5.3]{Schneider}.
\begin{fact}\label{Fact2}
    Let $K$ be a smooth body, then the support of its area measure $S_K$ is $S^{n-1}$.
\end{fact}
We recall lastly a standard  measure theory fact.
\begin{lemma}\label{Lem6}
    Let $\mu_i$ be a sequence of Borel measures on $S^{n-1}$ and $f_n$ a sequence of functions. Assume $\mu_i$ converges weakly to a finite measure $\mu$ and $f_i$ converges uniformly to a continuous function $f$. Then $\int_{S^{n-1}} f_i d\mu_i $ converges to $\int_{S^{n-1}} f d\mu$.
\end{lemma}
\begin{proof}
    Indeed, 
    \begin{align*}
         \lim_{i\to \infty}\Bigr|&\int_{S^{n-1}} f_i d\mu_i -\int_{S^{n-1}} f d\mu\Bigr|\\
         &\leq
         \limsup_{i\to \infty}
    \Bigr|\int_{S^{n-1}} f_i-f d\mu_i \Bigr |
    + 
    \limsup_{i\to \infty}\Bigr|\int_{S^{n-1}} f d\mu_i -\int_{S^{n-1}} f d\mu\Bigr| \\
    &\leq \limsup_{i\to \infty}\|f_i-f\|_{L^\infty(S^{n-1})}\mu_i(S^{n-1})=0,
    \end{align*}
    where we used the weak convergence of $\mu_i$ to $\mu$ in the second inequality.

\end{proof}

\subsection{}
We are now ready to state the analogous equation to \eqref{Eqn2bis} for   parallel bodies $(K_t)_{t}$ obtained from $K$ and a segment $I$.
\begin{proposition}
\label{Prop2}
Let $K\subseteq \R^n$ be a smooth body containing the origin in its interior such that $h_K|_{S^{n-1}}$ is  $C^2$. Let $I$ be an origin-symmetric segment.  Define $K_t=K+tI$ for $t>0$ and $K_t=K\div (-t)I$ for $t\leq 0$.
Let $f$ be a $C^2$ function on $S^{n-1}$.
Then,
    \begin{align}\nonumber
        \frac{d}{dt}\Bigr|_{t=0}&\Delta(K_t,f)\\
        \nonumber
        =&  \frac{(n-1)^2}{n} \left(  \frac{V(K[n-2],I,f)^2}{V(K[n-1],I)} -\frac{n-2}{n-1} V(K[n-3],I,f,f)-\frac{1}{n-1} V(K[n-2],I,\tfrac{f^2}{h_{K}}) \right)\\
        \nonumber
        &- n V(K[n-1],I)\left( \frac{V(K[n-1],f)}{V(K)} - \frac{n-1}{n} \frac{V(K[n-2],I,f)}{V(K[n-1],I)}\right)^2 \\    
        \label{Eqn2}
        &+\frac{1}{n} V(K[n-1],\tfrac{f^2}{h_{K}^2}h_I).
    \end{align}
\end{proposition}
\begin{proof}
    The formal algebraic computation is the same as \eqref{Eqn2bis}. Hence it remains to justify the interchange of mixed volume with taking derivatives. 
    
    \textbf{Step 1}:
    In this step we show $\frac{d}{dt}\Bigr|_{t=0} V(K_t)=n V(K[n-1],I)$.
    We have, 
    \begin{align}\nonumber
        \frac{d}{dt}\Bigr|_{t=0} V(K_t) 
    =& \lim_{t\to 0}\frac{V(K_t)-V(K)}{t}\\
    \nonumber
    =& \lim_{t\to 0} \sum_{i=1}^{n}V(K_t[n-i],K[i-1],\tfrac{h_{K_t}-h_K}{t} ) \\
    =&  \sum_{i=1}^{n}  \lim_{t\to 0}\frac{1}{n}
    \int_{S^{n-1}}\frac{h_{K_t}-h_K}{t} dS_{K_t[n-i],K[i-1]}.
    \label{Eqn26}
    \end{align}
     By Fact \ref{Fact1}, $\frac{h_{K_{t}}-h_{K}}{t}$ 
    converges uniformly to $h_I$. It follows that  $K_t$ converges to $K$ in Hausdorff distance as $t$ converges to 0. Hence  $S_{K_t[n-i],K[i-1]}$ weakly converges to $S_{K}$ \cite[Page 281]{Schneider}, which is a finite measure. 
    By Lemma \ref{Lem6},
    $$
     \lim_{t\to 0}
     \frac{1}{n}
    \int_{S^{n-1}}\frac{h_{K_t}-h_K}{t} dS_{K_t[n-i],K[i-1]} = V(K[n-1],h_I).
    $$
    Plug this back into \eqref{Eqn26}, we obtain 
    $\frac{d}{dt}\Bigr|_{t=0} V(K_t)=n V(K[n-1],h_I)$ as desired.

    Since $f$ is $C^2$, we may write $f=h_L-h_M$ for some convex bodies $L,M$ \cite[Cor.2.2]{bochner}. It is then clear that the above proof works \textit{verbatim} for $\frac{d}{dt}\Bigr|_{t=0} V(K_t[n-1],f)=(n-1) V(K[n-2],I,f)$ and $\frac{d}{dt}\Bigr|_{t=0} V(K_t[n-2],f,f)=(n-2) V(K[n-3],I,f,f)$.
    
    \textbf{Step 2}:
    In this step we show 
    \begin{equation}
        \label{Eqn31}
    \frac{d}{dt}\Bigr|_{t=0} V(K_t[n-1],\tfrac{f^2}{h_{K_t}})=(n-1)V(K[n-2],I,\tfrac{f^2}{h_{K}})-V(K[n-1],\tfrac{f^2}{h_K^2}h_I).
    \end{equation}
    Indeed we have,
    \begin{align*} 
    &\frac{ V(K_t[n-1],\tfrac{f^2}{h_{K_t}})-  V(K[n-1],\tfrac{f^2}{h_{K}})}{t}\\
    &\quad =
    V(K_t[n-1],\tfrac{1}{t}(\tfrac{f^2}{h_{K_t}}-\tfrac{f^2}{h_{K}}))  +\tfrac{1}{t} 
    (V(K_t[n-1],\tfrac{f^2}{h_{K}})-V(K[n-1],\tfrac{f^2}{h_{K}})).
    \end{align*}
    On the one hand, we have $\tfrac{1}{t}(\tfrac{f^2}{h_{K_t}}-\tfrac{f^2}{h_{K}}) = -\frac{f^2}{h_Kh_{K_t}} \frac{h_{K_t}-h_K}{t}$. 
    Since $\frac{h_{K_t}-h_K}{t}$  converges uniformly to $h_I$ as $t$ goes to $0$ by   Fact \ref{Fact1},
     we know $\tfrac{1}{t}(\tfrac{f^2}{h_{K_t}}-\tfrac{f^2}{h_{K}})$ converges uniformly to $-\frac{f^2}{h_K^2}h_I$. Also $S_{K_t}$ converges weakly to $S_K$.
    By Lemma \ref{Lem6},
    \begin{align}\label{Eqn30}
    \lim_{t\to 0}
         V(&K_t[n-1],\tfrac{1}{t}(\tfrac{f^2}{h_{K_t}}-\tfrac{f^2}{h_{K}})) 
         =-V(K[n-1],\tfrac{f^2}{h_K^2}h_I),
    \end{align}
    On the other hand, since $f$ and $h_K$ are $C^2$ functions and $K$ contains the origin in its interior, $\frac{f^2}{h_K}$ also is a $C^2$ function. Therefore, it is possible to find two convex bodies $L,M$ such that $\frac{f^2}{h_K}=h_L-h_M$ \cite[Cor.2.2]{bochner}.
    We can again run the exact same argument as in Step 1 to obtain 
    \begin{align*}
        \lim_{t\to 0}\tfrac{1}{t} &
    (V(K_t[n-1],\tfrac{f^2}{h_{K}})-V(K[n-1],\tfrac{f^2}{h_{K}}))\\
    &=\lim_{t\to 0}\bigr[\tfrac{1}{t} 
    (V(K_t[n-1],L)-V(K[n-1],L))-\tfrac{1}{t} 
    (V(K_t[n-1],M)-V(K[n-1],M))\bigr]\\
    &=(n-1)V(K[n-2],I,L)-(n-1)V(K[n-2],I,M)\\
    &=(n-1)V(K[n-2],I,\tfrac{f^2}{h_{K}}).
    \end{align*}
    Combine this with \eqref{Eqn30}, we obtain \eqref{Eqn31}.
 
\end{proof}

\begin{proof}[Proof of Theorem \ref{Thm2}]
Let $ f=h_L-ch_K$ where 
    \begin{equation}
    \nonumber
         c =n \frac{V(K[n-1],L)}{V(K)} - (n-1)\frac{V(K[n-2],I,L)}{V(K[n-1],I)}.
    \end{equation}
    From the choice of $c$, it follows that 
    \begin{equation}\label{Eqn8}
         \frac{V(K[n-1], f)}{V(K)} - \frac{n-1}{n} \frac{V(K[n-2],I, f)}{V(K[n-1],I)}=0
    \end{equation}
    Fix any origin-symmetric segment $I$.
    Let 
    $$
    K_t:=\left\{
    \begin{array}{cc}
       K+tI,      & t\geq 0, \\
       K\div (-t)I , & t<0.
    \end{array}
    \right.
    $$
    Since $K$ is smooth, $I$ is adapted to $K$.
    Using Fact \ref{Fact1}, $\frac{d}{dt} \bigr|_{t=0}h_{K_t}  =  h_I$.
    By Proposition \ref{Prop2},
    \begin{align}\nonumber
        \frac{d}{dt}\Bigr|_{t=0}&\Delta(K_t, f)\\
        \nonumber
        =& \frac{(n-1)^2}{n^2} |I| \Delta(P_{I^\perp}K,f|_{I^\perp})
        +\frac{1}{n} V(K[n-1],\tfrac{ f^2}{h_{K}^2}h_I)\\
        \nonumber
        &- n V(K[n-1],I)\left( \frac{V(K[n-1], f)}{V(K)} - \frac{n-1}{n} \frac{V(K[n-2],I, f)}{V(K[n-1],I)}\right)^2 \\ 
        \label{Eqn11}
        \geq & \frac{1}{n} V(K[n-1],\tfrac{ f^2}{h_{K}^2}h_I)\geq 0,
    \end{align}
    where for the first inequality we used the assumption that the local log-Brunn-Minkowski inequality is true and \eqref{Eqn8}.
    
    By the assumption on LLBM we have $\Delta(K_t, f)\geq 0$. But  Lemma \ref{Lem1}  implies $\Delta(K, f)=0$. Hence $t=0$ minimizes $\Delta(K_t, f)$ and we have
    \begin{equation}\nonumber
        \frac{d}{dt} \Bigr|_{t=0}\Delta(K_t, f)=0.
    \end{equation}
    Combine this with \eqref{Eqn11}, we obtain $V(K[n-1],\tfrac{ f^2}{h_{K}^2}h_I)=0$, which is 
    $$
    \int_{S^{n-1}} \frac{ f^2}{h_{K}^2}h_I dS_K=0.
    $$
    Therefore, 
    $$
    \frac{ f^2}{h_{K}^2}h_I = 0 \text{ on } \opn{supp} S_{K}.
    $$
    By Fact \ref{Fact2}, $\opn{supp} S_K=S^{n-1}$. 
    Since $h_I(u)\neq 0$ holds for all $u\notin I^\perp$, we have $ f(u)=0$ for all $u\notin I^\perp\cap S^{n-1}$. Since $ f$ is continuous and $S^{n-1}\setminus I^\perp$ is dense in $S^{n-1}$, we conclude that $ f(u)=0$ for all $u\in S^{n-1}$. 
    Therefore, $h_L=ch_K+ f=ch_K$, which completes the proof.

\end{proof}

 \bibliographystyle{plain}
 \bibliography{myref.bib}

\end{document}